\documentclass[12pt]{amsart}

\usepackage[OT4]{fontenc}
\usepackage{graphics,graphicx}
\usepackage{amsmath}
\usepackage{amssymb,amsfonts}
\usepackage{amsthm}
\usepackage{xcolor}
\usepackage{enumerate}
\usepackage{hyperref}
\usepackage[cp1251]{inputenc}
\usepackage{cleveref}
\usepackage{lipsum}
\usepackage[active]{srcltx}
\usepackage{mathtools}
\newtheorem{theorem}{Theorem}

\newtheorem{corollary}{Corollary}
\usepackage{etoolbox}
\appto\normalsize{\belowdisplayshortskip=\belowdisplayskip}
\appto\normalsize{\abovedisplayshortskip=\abovedisplayskip}
\newcommand\myfootnote[1]{
\renewcommand{\thefootnote}{}
\footnotetext{#1}
\def\thefootnote{\@arabic\c@footnote}
}

\makeatletter
\renewcommand{\subsection}{\@startsection{subsection}{2}{0mm}{-\baselineskip}{-5pt}{\it \bf}}
\makeatother

\title{$\sigma$-S\MakeLowercase{table} M\MakeLowercase{atrices}}
\author{M\MakeLowercase{ichael} A.S. T\MakeLowercase{horne}}
\date{}

\begin{document}

\address{B\MakeLowercase{ritish} A\MakeLowercase{ntarctic} S\MakeLowercase{urvey}, C\MakeLowercase{ambridge}}

\email{mior@bas.ac.uk}

\begin{abstract}
$\sigma$-Stable matrices are introduced 
and it is shown that the real roots of the polynomials comprising the coefficients of the characteristic polynomial indicate the coefficient sign changes. A proof of Obrechkoff is then used to show that the largest real root from the coefficients is the point of stability when the maximal eigenvalue of the $\sigma$-stable matrix is in $\mathbb{R}$. Some implications of the coefficient behaviour for a scaling relation are then discussed. 
\end{abstract}

\subjclass{93D15}
        
%\keywords{stability; characteristic polynomial}
\maketitle

Let $\mathbf{M}$ be an $n \times n$ square matrix with real-valued entries, $m_{i,j} \in \mathbb{R}$ $\forall$ $i,j$. Then, multiply the diagonal entries $\mathbf{D}$ of $\mathbf{M}$ by a variable $\sigma \in \mathbb{C}$,\\

\begin{equation*}
\mathbf{M}_{\sigma}=\mathbf{M}-(1-\sigma)\cdot \mathbf{D} = 
\begin{pmatrix}
\sigma m_{1,1} & m_{1,2} &  \cdots & m_{1,n} \\
m_{2,1} & \sigma m_{2,2} & \cdots & m_{2,n} \\
m_{3,1} & m_{3,2} &  \cdots & m_{3,n} \\
\vdots  & \vdots  & \ddots  & \vdots  \\
m_{n,1} & m_{n,2} & \cdots & \sigma m_{n,n} 
\end{pmatrix}.
\label{eq:sigma}
\end{equation*}\\

If a $\sigma > 0$ can be found for $\mathbf{M}_{\sigma}$ such that the maximal (largest real-part) eigenvalue is $0$, $\operatorname{Re}(\lambda_{max(\mathbf{M}_{\sigma})}) = 0$, and any larger real-part of $\sigma$ results in $\operatorname{Re}(\lambda_{max(\mathbf{M}_{\sigma})}) < 0$, then $\mathbf{M}_{\sigma}$ is $\sigma$-stable. 

The characteristic polynomial, $P$, of $\mathbf{M}_{\sigma}$ is a monic polynomial in $x$, each of whose coefficients is a polynomial in $\sigma$:

\begin{equation*}
  P_{\mathbf{M}_{\sigma}} = \sum_{i=0}^n p_i x^i = \sum_{i=0}^n (\sum_{j=0}^{n-i} q_j\sigma^j) x^i
\end{equation*}

with $p_n=1$ and 

\begin{equation*}
p_i = \sum_{j=0}^{n-i} q_j\sigma^j
\end{equation*}\\

where the $q_j$ consist of terms of length $n-i$ from elements of $\mathbf{M}$. In each $p_i$, the $(n-i-1)^{st}$ term in the polynomial is $0$, with other terms also potentially absent (but never the $(n-i)^{th}$ term). If each $p_i$ has a leading coefficient that is positive, then $\mathbf{M}_{\sigma}$ can be made $\sigma$-stable through diagonal dominance [1]. This can occur for the highest order of $\sigma$ in each $p_i$ if each of the sums of the $\binom{n}{i}$ sets of the negative of the diagonal elements $\{\sum_{i=1}^n (-m_{i,i}),\sum_{i \neq l} (-m_{i,i})(-m_{l,l}),...\}$ are positive.

The polynomials $p_i$ provide an indication of the behaviour of each coefficient in the characteristic polynomial of $\mathbf{M}_{\sigma}$ depending on the value of $\sigma$ in relation to the individual real-valued roots of each $p_i$ ($\sigma_{i,1}...\sigma_{i,k}$ for some $k$ specific to each $p_i$). 

\begin{theorem}
The real-valued roots $({\sigma_{i,1},...,\sigma_{i,k}} \in \mathbb{R})$ of each $p_i$ indicate the point of respective coefficient sign change in the characteristic polynomial given the value of $\sigma$ of $\mathbf{M}_{\sigma}$. Furthermore, the largest real-valued root for each $p_i$ indicates when the coefficient becomes positive and it remains positive for any larger value of $\sigma$ (restricted to $\mathbb{R}$) applied to $\mathbf{M}_{\sigma}$.   
\label{theorem}
\end{theorem}

\begin{proof}
A necessary condition for stability is that all the coefficients of a characteristic polynomial are positive (a point that can be deduced for  $\sigma$-stable matrices through the sign of the trace and Descartes' rule of signs [2]). Further, any $\sigma$-stable matrix can be made stable through diagonal dominance. This implies that eventually, given the appropriate $\sigma$, all the coefficients will become positive. Each coefficient of the characteristic polynomial in $x$, which is in the field of $\mathbb{R}$, is itself a polynomial, $p_i$, in $\sigma$ and, again by Descartes' rule of signs, the real roots reflect the changing sign of the $p_i$ and therefore of the coefficients of the characteristic polynomial in $x$. Diagonal dominance ensures that any $\sigma$ larger than the largest real root of each $p_i$ will produce a positive coefficient. Any and all smaller real roots of $p_i$ indicate the points at which, for the appropriate $\sigma$, the sign of the coefficient alternates. 
\end{proof}

While a necessary condition for stability is that all the coefficients of a characteristic polynomial are positive, it is not sufficient. This point, along with Theorem~\ref{theorem}, leads to the following theorem and corollary. For these, we define the set 
$\Omega=\{\{\sigma_{0,1},...,\sigma_{0,k_0}\},...,\{\sigma_{i,1},...,\sigma_{i,k_i}\},...,\\\{\sigma_{n-1,1},...,\sigma_{n-1,k_{n-1}}\}\}$ consisting of all the real-valued roots ($\Omega \in \mathbb{R}$) of all the $p_i$ polynomials from the characteristic polynomial of a given $\mathbf{M}_{\sigma}$.

\begin{theorem}
If, for a given $\sigma$, the leading eigenvalue, $\lambda_{max(\mathbf{M}_{\sigma})}$, of $\mathbf{M}_{\sigma}$ is $0$, and $\lambda_{max(\mathbf{M}_{\sigma})} \in \mathbb{R}$ ($\operatorname{Re}(\lambda_{max(\mathbf{M}_{\sigma})}) = 0$, $\operatorname{Im}(\lambda_{max(\mathbf{M}_{\sigma})}) = 0$), then $max(\Omega) = \sigma$.
\label{lemreal}
\end{theorem}

\begin{proof}
By the necessary condition of stability, all coefficients must be positive, which means that $\sigma \ge max(\Omega)$. 
By the proof of Obrechkoff [3], if a polynomial has positive coefficients and has a root in the right-half plane, then the root cannot lie on the real axis. This implies that if $\sigma \in \mathbb{R}$ then $\sigma \leq max(\Omega)$. Therefore, $\sigma = max(\Omega)$.
\end{proof}

\begin{corollary}
If, for a given $\sigma$, the real-part of the leading eigenvalue, $\lambda_{max(\mathbf{M}_{\sigma})}$, of $\mathbf{M}_{\sigma}$ is $0$ but with a non-zero imaginary part ($\operatorname{Re}(\lambda_{max(\mathbf{M}_{\sigma})}) = 0$, $\operatorname{Im}(\lambda_{max(\mathbf{M}_{\sigma})}) \neq 0$), then $max(\Omega) \leq \sigma$.
\end{corollary}

\begin{proof}
By the necessary condition of stability, all coefficients of the characteristic polynomial must be positive. Therefore $max(\Omega) \leq \sigma$.
\end{proof}

A scaling operation, first introduced in [4], related the $\sigma$ of $\mathbf{M}_{\sigma}$ with the maximal eigenvalue of a matrix, $\bar{\mathbf{M}}_0$, defined by

\begin{equation*}
\bar{\mathbf{M}}_0 = -\mathbf{D}^{-1} \mathbf{M}+\mathbf{I}.
\end{equation*}

The characteristic polynomial of $\bar{\mathbf{M}}_0$ in $x$ is the same polynomial as that for the last coefficient $p_0$ in $\sigma$ of the characteristic polynomial of $\mathbf{M}_{\sigma}$ ($P_{\bar{\mathbf{M}}_0} = p_0$, $p_0 = (-1)^n$ det$(\mathbf{M}_{\sigma})$). Therefore when $\sigma=\lambda_{max(\bar{\mathbf{M}}_0)}$, $p_0$ vanishes, forcing $\mathbf{M}_{\sigma}$ to have an eigenvalue of $0$. \\

\centerline {\bf References}
\medskip

\noindent [1] S. Gerschgorin, \"{U}ber die Abgrenzung der Eigenwerte einer Matrix. Izv. Akad. Nauk. USSR Otd. Fiz.-Mat, 6: 749-754, 1931.
\medskip

\noindent [2] R. Descartes, La G\'{e}om\'{e}trie. Paris, 1637. 
\medskip

\noindent [3] N. Obrechkoff, Sur un probl\`{e}me de Laguerre. C.R. Acd. Sci. Paris, 177: 223-235, 1923.
\medskip

\noindent [4] A-M. Neutel, M. Thorne, Interaction strengths in balanced carbon cycles and the absence of a relation between ecosystem complexity and stability. Ecology Letters, 17 (6): 651-661, 2014.

\end{document}